\documentclass[10pt]{article}

\usepackage{amsthm}

\usepackage{amsmath,amssymb}
\usepackage{authblk}
\usepackage{csvsimple}
\usepackage{enumerate}
\usepackage{enumitem}
\usepackage{float} 
\usepackage[right=2.5cm,top=2.5cm,bottom=2.5cm,left=2.5cm]{geometry} 
\usepackage{graphicx} 
\usepackage{booktabs}
\usepackage[utf8]{inputenc}
\usepackage{subcaption}
\usepackage{siunitx}
\usepackage[colorlinks=true,citecolor=magenta]{hyperref}
\usepackage{subcaption} 
\usepackage{url}
\usepackage{setspace} 
\usepackage{tikz}
\usepackage{mathrsfs}
\usepackage{bm}
\usepackage{cite}
\usepackage{pgfplots}
\usepackage{multicol,multirow}
\pgfplotsset{compat=newest}

\newtheorem{prop}{Proposition}
\newtheorem{theorem}{Theorem}
\newtheorem{cor}{Corollary}
\theoremstyle{definition}
\newtheorem{definition}{Definition}

\title{\bf Sequence of Polyhedral Relaxations \\ for Nonlinear Univariate Functions}
\author[1]{Kaarthik Sundar}
\affil[1]{\textit{Los Alamos National Laboratory, New Mexico, USA}, \texttt{\{kaarthik,harsha\}@lanl.gov}}
\author[2]{Sujeevraja Sanjeevi}
\affil[]{\textit{Sabre, Southlake, Texas, USA}, \texttt{sujeev.sanjeevi@gmail.com}}
\author[1]{Harsha Nagarajan}
\providecommand{\keywords}[1]{\textbf{\textit{Index terms---}} #1}

\date{}

\begin{document}
\maketitle 
\begin{abstract}
    Given a nonlinear, univariate, bounded, and differentiable function $f(x)$, this article develops a sequence of Mixed Integer Linear Programming (MILP) and Linear Programming (LP) relaxations that converge to the graph of $f(x)$ and its convex hull, respectively. Theoretical convergence of the sequence of relaxations to the graph of the function and its convex hull is established. For nonlinear non-convex optimization problems, the relaxations presented in this article can be used to construct tight MILP and LP relaxations. These MILP and the LP relaxations can also be used with MILP-based and spatial branch-and-bound based global optimization algorithms, respectively.

    \noindent \keywords{piecewise relaxation; univariate function; global optimization; convex hull; continuous relaxation}
\end{abstract}

\section{Introduction} \label{sec:intro}
Developing convex relaxations of non-convex functions is a core part of any global optimization algorithm for nonlinear, non-convex optimization problems \cite{Tawarmalani2002}. Traditional global optimization algorithms rely on convex relaxations of non-convex functions to construct a relaxation of the optimization problem and obtain bounds on the optimal objective value with very little computational effort. These relaxations are then embedded into a spatial branch-and-bound (sB\&B) framework to obtain a global optimal solution \cite{Smith1997,Pardalos2013,Floudas2014}. The past decade has seen the advent of MILP-based global optimization algorithms \cite{Wicaksono2008,Misener2011,Teles2013,Castillo2018,Nagarajan2019,Nagarajan2016,Dambrosio2012} that solve a sequence of MILPs or convex MINLPs based on successively tighter piecewise convex relaxations to compute the optimal solution. The main motivation of this approach is the meteoric improvement in the speed of off-the-shelf MILP solvers \cite{GurobiBenchmarks}. Both sB\&B-based and MILP-based global optimization algorithms have one core sub-problem each. The sB\&B-based algorithms rely on constructing and solving convex relaxations of non-convex structures and MILP-based algorithms rely on solution to piecewise convex relaxations to non-convex structures. Though commercial sB\&B-based solvers like BARON \cite{baron} are fairly mature, they still lack support for a wide variety of functions like trigonometric, hyperbolic etc. Similarly, given the nascency of MILP-based global optimization methods \cite{Misener2011,Castro2016,Nagarajan2019}, their support for non-convex structures is even lesser. In particular, \cite{Misener2011,Castro2016,Nagarajan2019} have extensively looked at polynomial optimization problems and developed many piecewise convex relaxations for a single multilinear term. For other polynomial terms, these papers have relied on piecewise convex relaxations of the equation $y = x^2$ and the idea of factorable functions to build piecewise relaxations of other (higher-order) non-convex polynomial terms. 

Before we present literature related to this work, we illustrate the importance of developing polyhedral and piecewise polyhedral relaxations for univariate functions in global optimization algorithms. To this end, we consider the $1730$ NLP instances in MINLPLib. A total of $20$ nonlinear functions (see Table 3 in \url{https://www.gams.com/latest/docs/UG_Parameters.html} for their definitions) occur in these instances. The number of instances that contain each of these functions are shown in Fig. \ref{fig:operators}. Among the $20$ functions, all but `mul', `div', `rpower', `centropy', `mod', and `min' are unary nonlinear functions i.e. univariate functions.   
\begin{figure}
    \centering
    \includegraphics[scale=0.9]{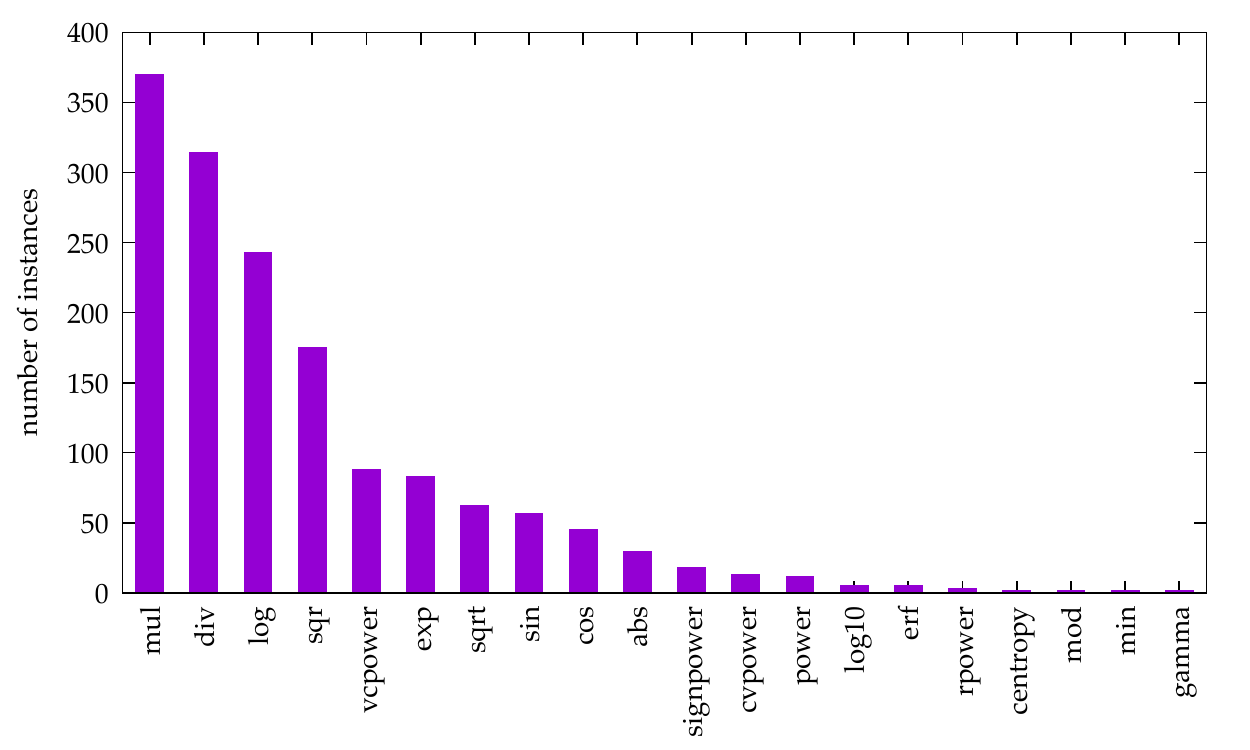}
    \caption{Number of instances in MINLPLib where each nonlinear function occurs at least once.}
    \label{fig:operators}
\end{figure}
Given any factorable MINLP, i.e. each nonlinear function in the MINLP is a factorable function, the MINLP can be equivalently reformulated to the following problem via introduction of sufficient number of artificial variables:
\begin{subequations}
\begin{flalign}
    \min: \quad & \bm c \cdot \bm x  + \bm d \cdot \bm y \label{eq:obj-minlp} \\
    \text{subject to:} \quad & A_1 \bm x + A_2 \bm y \leqslant \bm {b},  \label{eq:linear-minlp} \\ 
    & y_i = f_i(\bm x) \quad \forall f_i \in \text{ univariate functions}, \label{eq:univariate-minlp}  \\
    & y_j = g_j(\bm x) \quad \forall g_j \in \text{ multilinear functions}, \label{eq:multi-minlp} \\ 
    & \bm{l} \leqslant \bm{x} \leqslant \bm{u}, \label{eq:bounds-minlp} \\ 
    & \bm{x} \in \mathbb R^{n_1} \times \mathbb Z^{n_2} \times \{0, 1\}^{n_3}. \label{eq:type-minlp}
\end{flalign}
\label{eq:minlp}
\end{subequations}
To solve the MINLP in Eq. \eqref{eq:minlp} to global optimality using a sB\&B-based method, a convex relaxation is required for each Eq. in \eqref{eq:univariate-minlp} and \eqref{eq:multi-minlp}. On the other hand, for a MILP-based method, a piecewise convex relaxation for each Eq. in \eqref{eq:univariate-minlp} and \eqref{eq:multi-minlp}, given partitions for a subset of variables in the decision variable vector $\bm{x}$, is required \cite{Misener2011,Castro2016,Nagarajan2019}. In the literature, authors in \cite{Castro2016, Nagarajan2019, Sundar2020} have developed many piecewise convex relaxations for multilinear functions and $y = x^2$ i.e., `mul', `sqr', and `power'. This has enabled very efficient MILP-based global solution methods that are comparable to state-of-the-art spatial branch-and-bound solvers like BARON to certain classes of problems like polynomial optimization problems \cite{Nagarajan2019}. Extending these methods to other problems in the MINLPLib necessitates development of piecewise relaxations for the other non-convex functions shown in Fig. \ref{fig:operators}. All functions other than the ones mentioned above lack piecewise convex relaxations. On the other hand, developing convex relaxations for general univariate functions would aid sB\&B-based solvers like BARON to support a wider family of nonlinear non-convex functions. This article focuses on nonlinear univariate functions that are bounded and differentiable in their domains, and develops both polyhedral relaxations and piecewise polyhedral relaxations for these functions that can directly {be integrated into sB\&B-based and MILP-based global optimization approaches, respectively.} This class of functions is vast, and subsumes $14$ of the $20$ nonlinear functions in Fig. \ref{fig:operators}. Furthermore, `abs' and `min' which do not fall under the category of functions in this article, have disjunctive and convex reformulations that can directly be integrated into existing global methods. In summary, given a nonlinear, bounded, and differentiable univariate function $y = f(x)$ with a bounded domain $[\ell, u]$, the contributions of this article are as follows: (i) we develop a sequence of MILP and LP relaxations that converge to the graph of $f(x)$ and its convex hull, respectively. (ii) For a given partition of the domain, we construct an MILP and an LP relaxation that can be directly utilized in MILP-based and sB\&B-based global optimization algorithms. 

The idea of developing polyhedral and piecewise polyhedral relaxations for nonlinear, univariate functions is not new and has been addressed in the literature \cite{Gounaris2008,Liberti2003,Kallrath2014,Rebennack2016}. The difference between a polyhedral relaxation and a piecewise polyhedral relaxation is that for the latter, the domain of the function is partitioned into sub-domains, a polyhedral relaxation is constructed for each sub-domain and a disjunctive formulation to capture the relaxation for the whole domain is obtained. A polyhedral relaxation, on the other hand, constructs a polyhedron that contains the entire function in its domain. Hence, by definition, a piecewise polyhedral relaxation is an MILP. The authors in \cite{Gounaris2008} develop a series of tight convex $\alpha$-BB under-estimators for univariate functions in $\mathcal C^2$ and prove that it converges to the convex hull on one side of the function with increasing partition density. In \cite{Liberti2003}, the authors focus on developing convex envelopes for monomial terms with odd-degree. Authors in \cite{Kallrath2014} develop piecewise relaxations for univariate functions by solving multiple non-convex MINLPs to global optimality. Finally, in \cite{Rebennack2016}, the author develops piecewise over- and under-estimators for the function $x^2$, and these estimators are essentially the same as the polyhedral relaxations developed in this article. Our work differs from the aforementioned works in the sense that (i) it develops a single framework for all nonlinear, univariate, bounded, and differentiable functions (i.e., $\mathcal C^1$) which provides both MILP (piecewise polyhedral) and LP (polyhedral) relaxations for a fixed partition of the domain. (ii) It has theoretical convergence properties for both the LP and the MILP relaxations with increasing partition density, similar to \cite{Gounaris2008}. The theoretical convergence of the LP relaxations to the convex hull is a novel result and, to the best of our knowledge, is the first result that takes a step in the direction of characterizing the convex hull of a nonlinear, non-convex, bounded, and differentiable univariate functions. (iii) It is valid for a wider class of univariate, nonlinear functions unlike \cite{Liberti2003,Rebennack2016}. (iv) It can easily be integrated into state-of-the-art global optimization approaches, be it sB\&B-based or MILP-based. 

The remainder of the article is organized as follows. Sec. \ref{sec:statement} introduces the formal problem statement and notations. Sec. \ref{sec:geometry} introduces the main intuition behind the relaxations geometrically and presents the formulations. Sec. \ref{sec:sequence} presents the theoretical properties of the relaxations. Finally, Sec. \ref{sec:results} presents the computational results with conclusions in Section \ref{sec:conclusion}.

\section{Problem statement} \label{sec:statement}
We are given a nonlinear, univariate function $f: [\ell, u] \rightarrow \mathbb R$ that is bounded and differentiable. The objective of the article is to develop a sequence of MILP relaxations converging to the graph of $f(x)$ i.e., $\mathcal X = \{ (x, y) \in [\ell, u] \times \mathbb{R} : y = f(x) \}$, with corresponding LP relaxations converging to the convex hull of $\mathcal{X}$, $\operatorname{conv}(\mathcal X)$. Throughout the rest of the article, we always assume the function under investigation is nonlinear, univariate, bounded, and differentiable. We first introduce some definitions and notations that will be used throughout the rest of the article. 
\begin{definition} \label{def:partition} \textit{Partition of a closed interval} --
 Given a closed interval $[\ell, u] \subset \mathbb{R}$, a partition $p$ of $[\ell, u]$ is an ordered sequence of real numbers $(x_0,x_1,\ldots,x_k)$ such that $\ell=x_0<x_1<\ldots<x_k=u$. We denote the set of partition points of $p$ by $\mathcal P(p)$, i.e., $\mathcal P(p) = \{x_0, \dots, x_k\}$, and the set of all possible partitions of $[\ell, u]$ by $\mathcal A[\ell, u]$. Given $p \in \mathcal A[\ell, u]$, each closed interval $[x_i, x_{i+1}]$ is referred to as a sub-interval of $p$, and we let $\mathscr S(\cdot)$ be the function mapping $p$ to its set of sub-intervals.
\end{definition}
\begin{definition} \label{def:break-point} \textit{Break point of a function} -- Given a nonlinear, bounded, and differentiable function $f: [\ell, u] \rightarrow \mathbb R$, $b \in  [\ell, u]$ is referred to as a break point of the function, if at $b$, the function $f$ changes from being convex to concave or vice versa. 
\end{definition}
\begin{definition} \label{def:base-partition} \textit{Base partition of a function} --
Given a function $f: [\ell, u] \rightarrow \mathbb R$, the partition $p^0_f \in \mathcal A[\ell, u]$ is referred to as the base partition of $f$ if it satisfies: (i) every break point of $f$ in $[\ell, u]$ is in $\mathcal{P}(p)$, (ii) for any sub-interval $[x_i, x_{i+1}] \in \mathscr S (p^0_f)$, $f'(x_i) \neq f'(x_{i+1})$, and (iii) $|\mathcal P(p^0_f)|$ is minimum.
\end{definition}
\noindent The above definition makes $\mathcal P(p^0_f)$ non-unique due to condition (ii). To make it unique, without loss of generality, we construct $p^0_f$ with a partition that includes only the break points of $f$ in its domain, and add a partition point at the mid-point of every sub-interval $[x_i, x_{i+1}]$, if $f'(x_i) = f'(x_{i+1})$. The inclusion of break points in the definition of $p^0_f$ ensures that in each sub-interval of $p^0_f$, the function $f$ is either concave or convex. 
\begin{definition} \label{def:refinement} \textit{Refinement of a partition} --
Given two partitions $p, q \in \mathcal A[\ell, u]$, $q$ is defined to be the refinement of $p$, if $\mathcal P(p) \subset \mathcal P(q)$. 
\end{definition}
\noindent Any function that refines a partition is also referred to as a refinement scheme. In this article, we will present all the properties of the forthcoming polyhedral relaxations of $f$ with respect to the interval-bisection refinement scheme defined below. 
\begin{definition} \label{def:interval-bisection} \textit{Interval bisection} --
Given any partition $p \in \mathcal A[\ell, u]$, a refinement scheme $\mathscr B: \mathcal A[\ell, u] \rightarrow \mathcal A[\ell, u]$ is referred to as interval bisection if is satisfies the following properties: (i) $\mathcal P(p) \subset \mathcal P(\mathscr B(p))$, $|\mathcal P(\mathscr B(p))| = |\mathcal P(p)| + 1$  and (ii) the additional partition point in $\mathcal P(\mathscr B(p)) \setminus \mathcal{P}(p)$ is a mid-point of some sub-interval in $p$ i.e., $\mathscr S(p)$.
\end{definition}
\begin{definition} \label{def:admissible-partitions} \textit{Admissible partitions for a function} -- 
Given $f: [\ell, u] \rightarrow \mathbb R$ and a partition $p \in \mathcal A[\ell, u]$, $p$ is defined to be admissible to $f$ if $p = p^0_f$, or if $p$ is obtained by applying a sequence of interval bisections on $p_f^0$. The set of all admissible partitions of $f$ is denoted by $\mathcal A_f[\ell, u]$. 
\end{definition}
Next, we present a systematic way to construct base MILP and LP relaxations for a function $f: [\ell, u] \rightarrow \mathbb R$ and an admissible partition $p \in \mathcal A_f[\ell, u]$. 

\section{Geometry of polyhedral relaxations} \label{sec:geometry}
An MILP relaxation for the function $f: [\ell, u] \rightarrow \mathbb R$ is constructed using an admissible partition $p \in \mathcal A_f[\ell, u]$. The relaxation constructed using $p_f^0$ is referred to as the base MILP relaxation of $f$. The core idea of the relaxation is: if $f$ is convex or concave in any closed interval, construct a triangle whose one side is the secant of $f$ through endpoints of the interval, and the other two sides are the tangents of $f$ at the endpoints; the region bounded by this triangle is a trivial relaxation of $f$ in that closed interval. Given $p_f^0$, the base MILP relaxation is obtained as the disjunctive union of a chain of triangles, one for each sub-interval in $\mathscr S(p_f^0)$. This construction is guaranteed to result in a relaxation of $f$ since the definition \ref{def:base-partition} ensures that $f$ is either concave or convex in each sub-interval. It is also guaranteed to be a relaxation for any admissible partition $p \in \mathcal A_f[\ell, u] \setminus \{p_f^0\}$ since $p$ has to necessarily be some refinement of $p_f^0$. Given a sub-interval $[x_i, x_{i+1}] \in \mathscr S(p_f^0)$, the equations of the two tangents and the secant represented as a piecewise linear function and linear equation, respectively, are as follows:
\begin{subequations}
\begin{flalign}
h(x_i) &= f(x_i) + f'(x_i) \cdot (x - x_i) \label{eq:tangent-1} \\
t_{i+1}(x) &= 
\begin{cases}
    \max ~ \{ h(x_i), h(x_{i+1}) \} & \text{ if $f(x)$ is convex in $[x_i, x_{i+1}]$ } \\ 
    \min ~ \{ h(x_i), h(x_{i+1}) \} & \text{ if $f(x)$ is concave in $[x_i, x_{i+1}]$ } \\ 
\end{cases} \label{eq:tangent-plf}\\
s_{i+1}(x) &= f(x_i) + \frac{f(x_{i+1}) - f(x_i)}{x_{i+1}-x_i} \cdot (x - x_i) \label{eq:secant}
\end{flalign}
\label{eq:bounding_f}
\end{subequations}
The Eq. \eqref{eq:tangent-1} is the equation of the tangent of the function $f$ at $x_i$. When $f$ is convex in the sub-interval $[x_i, x_{i+1}]$, $t_{i+1}(x)$ (resp. $s_{i+1}(x)$) is the lower (resp. upper) bounding function and when $f$ is concave in the sub-interval, $t_{i+1}(x)$ (resp. $s_{i+1}(x)$) becomes the upper (resp. lower) bounding function. The base relaxation obtained for the function $y=x^3$ in $[-1.5, 2]$ is shown in Fig. \ref{fig:x3-relaxation}. The function is convex in $[0, 2]$ and concave in $[-1.5, 2]$. Hence, in the sub-intervals $[0, 2]$ and  $[-1.5, 2]$ the sides of the triangle formed by the tangents form the lower and upper bounding functions, respectively. 
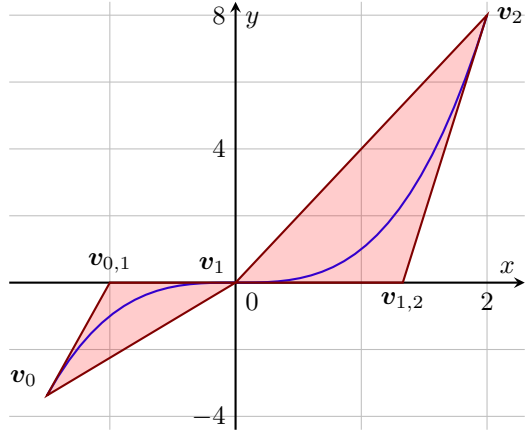
\begin{figure}[htbp]
\centering
\begin{tikzpicture}
\begin{axis}[ xlabel={$x$}, ylabel={$y$}
  ,axis lines=middle
  ,xmin=-1.8, xmax=2.3, ymin=-4.4, ymax=8.4
  ,samples=41, grid, thick
  ,domain=-1.5:2
  ,legend pos=outer north east
  ,xtick={}
  ,ytick={}
  ,ticks=none
  ]
  \addplot+[no marks] {x^3};
  \addplot [red!50!black,fill=red,fill opacity=0.2]
  coordinates 
  {(0, 0) (2, 8) (1.33, 0.0)} -- cycle;
  \addplot [red!50!black,fill=red,fill opacity=0.2]
  coordinates 
  {(-1.5, -3.375) (0, 0) (-1, 0.0)} -- cycle;
  \node [above left] at (-1.5, -3.375) {$\bm v_0$};
  \node [above left] at (0, 0) {$\bm v_1$};
  \node [right] at (2, 8) {$\bm v_2$};
  \node [above] at (-1, 0.0) {$\bm v_{0, 1}$};
  \node [below] at (1.33, 0.0) {$\bm v_{1, 2}$};
  \node [below right] at (0.0, 0.0) {$0$};
  \node [below] at (2, 0) {$2$};
  \node [left] at (0, 4) {$4$};
  \node [left] at (0, 8) {$8$};
  \node [left] at (0, -4) {$-4$};
 \end{axis}
\end{tikzpicture}
\caption{Base relaxation for $y=x^3$ for $x \in [-1.5, 2]$. The coordinates of the vertices are $\bm v_0 = (-1.5, -3.375)$,  $\bm v_1 = (0, 0)$,  $\bm v_2 = (2, 8)$,  $\bm v_{0, 1} = (-1, 0)$,  $\bm v_{1, 2} = (1.33, 0)$, and $p_f^0 = (-1.5, 0, 2)$. }
\label{fig:x3-relaxation}
\end{figure}
In the example shown in Fig. \ref{fig:x3-relaxation}, vertices $\bm v_0$, $\bm v_1$, and $\bm v_2$ are points on $y= f(x)$ and obtained using the base partition $p_f^0$. The vertex $\bm v_{i, i+1}$ for $i \in \{0, 1\}$ (in Fig. \ref{fig:x3-relaxation}) is obtained as the intersection of tangents at $\bm v_i$ and $\bm v_{i+1}$. Extrapolating this notation, given $f$ and $p \in \mathcal A_f[\ell, u]$, the vertices $\bm v_0, \bm v_1, \dots, \bm v_k$ ($k = |\mathcal P(p)|-1$) are points on the curve $y = f(x)$ at partition points of $p$ and the vertex $\bm v_{i, i+1}$ for $i \in \{0, \dots, k-1\}$ is the intersection of the two tangents to the curve at $\bm v_i$ and $\bm v_{i+1}$. The MILP relaxation for $\mathcal X$, denoted by $\mathcal R(f, p)$, is then given by the disjunctive union of a chain of $k$ triangles where the vertices of triangle $i$ are $\{\bm v_{i-1}, \bm v_i, \bm v_{i-1, i}\}$. We now define the strength of the relaxation $\mathcal R(f, p)$, given $f$ and $p$ as follows:
\begin{definition} \label{def:strength} \textit{Strength of $\mathcal R(f, p)$} --
Given $f: [\ell, u] \rightarrow \mathbb R$, $p \in \mathcal A_f[\ell, u]$ and the corresponding MILP relaxation $\mathcal R(f, p)$, the strength of the relaxation denoted by $\varepsilon$ is given by 
\begin{flalign}
\varepsilon(\mathcal R(f, p)) = \max_{i \in \{1, \dots k\}}  \max_{x \in [x_{i-1}, x_i]}  \left|t_i(x) - s_i(x)\right|  \label{eq:relaxation-strength}
\end{flalign}
where $k = |\mathcal P(p)| - 1$, and $t_i(x)$ and $s_i(x)$ are given by Eq. \eqref{eq:bounding_f}. 
\end{definition}
\noindent Informally, the relaxation strength $\varepsilon(\mathcal R(f, p))$ is defined as the maximum vertical distance over the disjunctive union of the $k$ triangles formed by secants and tangents at each sub-interval defined by the admissible partition $p$. We remark that the strength of the relaxation is always non-negative; lower the value of $\varepsilon(\mathcal R(f, p))$, the tighter the relaxation. Next, we formulate the disjunctive union of the chain of $k$ triangles as a MILP.

\paragraph{A formulation for $\mathcal R(f, p)$:} Again, we are given $f: [\ell, u] \rightarrow \mathbb R$ and an admissible partition $p \in \mathcal A_f[\ell, u]$. The number of triangles in $\mathcal R(f, p)$ is given by $k = |\mathcal P(p)| - 1$. Each triangle $i$ is associated with a binary variable $z_i$ and two non-negative continuous variables $\delta_1^i,\delta_2^i$. Given any vertex $\bm v$, we let $v^x$ and $v^y$ represent the $x$- and $y$-coordinate of $\bm v$. With this notation, we write the MILP formulation of the disjunctive union of triangles as: 
\begin{subequations}
\begin{flalign}
& y = v_0^y + \sum_{i = 1}^k \left\{ \delta_1^i(v_{i-1, i}^y - v_{i-1}^y) + \delta_2^i(v_{i}^y - v_{i-1}^y) \right\} \label{eq:y-triangle} \\
& x = v_0^x + \sum_{i = 1}^k \left\{ \delta_1^i(v_{i-1, i}^x - v_{i-1}^x) + \delta_2^i(v_{i}^x - v_{i-1}^x) \right\} \label{eq:x-triangle} \\
& \delta_1^1 + \delta_2^1 \leqslant 1 \label{eq:delta_bound} \\ 
& \delta_1^i + \delta_2^i \leqslant z_{i-1} \leqslant \delta_2^{i-1} \quad \forall i \in \{2, \dots, k\} \label{eq:fill-constraints} \\
& 0 \leqslant \delta_1^i, \delta_2^i \leqslant 1 \quad \forall i \in \{1, \dots, k\} \label{eq:cont} \\
& z_{i} \in \{0, 1\} \quad \forall i \in \{1, \dots, k-1\} \label{eq:bin}
\end{flalign}
\label{eq:triangle-chain}
\end{subequations}
The formulation \eqref{eq:triangle-chain} is referred to as the ``incremental formulation'' \cite{Yildiz2013} since the portion of each triangle is filled by the variables $x$ and $y$. The triangle $i+1$ can be filled only if the triangle $i$ is already filled i.e. when $\delta_2^i = 1$. While it is possible to develop other formulations to capture the disjunctive union of triangles with the same theoretical properties as the one in Eq. \eqref{eq:triangle-chain}, we refrain from doing so as that is not the focus of this article (see \cite{Yildiz2013} for details). We also remark that the MILP formulation for $\mathcal R(f, p)$ is in the variable space $(x, y, \bm \delta, \bm z)$ where $\bm \delta$ and $\bm z$ are vectors of additional continuous and binary variables in Eq. \eqref{eq:triangle-chain}. We will use the notation $\operatorname{Proj}_{x, y} \mathcal R(f, p)$ to denote the projection of $\mathcal R(f, p)$ to the $(x, y)$-variable space. It is not difficult to prove that $\operatorname{Proj}_{x, y} \mathcal R(f, p)$ is actually the disjunctive union of the $k$ triangles in the $(x, y)$ space.

\section{Sequence of relaxations and theoretical properties} \label{sec:sequence}
Before we present the sequence of polyhedral (MILP) relaxations for the function $f$, we invoke some results from literature that will be used to prove their theoretical properties. 
\begin{prop} (Lemma 2 in \cite{Thakur1980}) \label{prop:thakur}
Given the function $f:[\ell, u] \rightarrow \mathbb R$, a sub-interval $[a,b]$ of $[\ell, u]$ where $f$ is convex or concave, and the partition that consists of the points $a$ and $b$, i.e., $p = (a, b) \in \mathcal A[a, b]$, the strength of the relaxation $\mathcal R(f, p)$ is bounded above according to the following inequality:
\begin{flalign}
\varepsilon(\mathcal R(f, p) \leqslant \frac{ (b-a) |f'(a) - f'(b)|}4 \label{eq:triangle-error-bound}
\end{flalign}
\end{prop}
\begin{proof}
See \cite{Thakur1980}.
\end{proof}

\subsection{Algorithm for generating sequence of relaxations} \label{subsec:algo}
The sequence of MILP relaxations for the function $f: [\ell, u] \rightarrow \mathbb R$ is obtained using the following procedure. The base relaxation is given by $\mathcal R(f, p_0)$ where $p_0 = p_f^0$. A sequence of partitions $\{p_0, p_1, p_2, \dots \}$ is then created according to the following procedure: Given $p_k$ and $\mathcal R(f, p_k)$, the upper bound on the strength of the relaxation in every sub-interval $i \in \mathscr S(p_k)$ is computed using Eq. \eqref{eq:triangle-error-bound}; we let $\epsilon_i$ denote this bound for sub-interval $i$. The sub-interval $i$ with the maximum $\epsilon_i$ is then divided into two equal parts. This refinement of the sub-interval is an interval bisection and results in the partition $p_{k+1}$. This procedure results in an infinite sequence of partitions $p_0, p_1, \dots$ such that $p_0 = p_f^0$, $\mathcal P(p_k) \subset \mathcal P(p_{k+1})$ and $|\mathcal P(p_{k+1})| = |\mathcal P(p_{k})| + 1$ for every $k$. The infinite sequence of MILP relaxations is then given by $\mathcal R(f, p_k)$ for $k \in \mathbb Z_+$. We now state and prove the main theoretical properties of this sequence. 

\begin{prop} (Theorem 2.3 in \cite{Burkard1991})\label{thm:bukard}
Given $f: [\ell, u] \rightarrow \mathbb R$, $\delta >0$, and a sub-interval $[a,b]$ where $f$ is convex or concave, there exists a finite $\mathcal N \in \mathbb Z_+$ such that $\varepsilon(\mathcal R(f, (a, b)) < \delta$ for any $k \geqslant \mathcal N$; in particular 
\begin{flalign}
\mathcal N \leqslant \max \left(2, \frac 32 \sqrt{\frac{b-a}{2\delta}\cdot |f'(b) - f'(a)|} \right)
\end{flalign}
\end{prop}
\begin{proof}
See \cite{Burkard1991}.
\end{proof}

\begin{theorem} \label{thm:strength-convergence}
$\operatorname{Lim}_{k\rightarrow \infty} \varepsilon(\mathcal R(f, p_k)) = 0$.
\end{theorem}
\begin{proof}
Given $\delta > 0$, we need a finite $\mathcal M \in \mathbb Z_+$ such that for any $k \geqslant \mathcal M$, $\varepsilon(\mathcal R(f, p_k)) < \delta$. For each sub-interval in $i \in \mathscr S (p_f^0)$, let $\mathcal N_i$ be the finite number provided by proposition \ref{thm:bukard} for the given value of $\delta$. Setting $\mathcal M = \sum_i \mathcal N_i$ completes the proof. 
\end{proof}

\begin{cor} \label{cor:milp-convergence}
$\operatorname{Lim}_{k \rightarrow \infty} \,\operatornamewithlimits{Proj}_{x,y} \mathcal R(f, p_k) = \mathcal X$. \hfill \qedsymbol
\end{cor}

We now present some theoretical properties of the linear programming (LP) relaxation of \eqref{eq:triangle-chain}. To do so, we let $\mathcal R^{lp}(f, p)$ be the LP relaxation of $\mathcal R(f,p)$, and $\operatorname{Proj}_{x,y} \mathcal R^{lp}(f, p)$ the projection of $\mathcal R^{lp}(f, p)$ to the $(x,y)$ space. To prove some properties of this LP relaxation, we invoke the definition of a \textit{locally ideal} MILP formulation, introduced by Padberg and Rijal \cite{Padberg2012}: an MILP formulation of a set $\mathcal Y$ is said to be locally ideal if the vertices of its corresponding LP relaxation satisfy all required integrality conditions. Additionally, in any locally ideal MILP formulation of the set $\mathcal Y$, the projection of its LP relaxation to the original space of variables ($x$, $y$ in this case) is exactly the convex hull of $\mathcal Y$ i.e., $\operatorname{conv}(\mathcal Y)$ \cite{Jeroslow1984}.

\begin{prop}
The incremental MILP formulation in Eq. \eqref{eq:triangle-chain} i.e., the MILP formulation for $\mathcal R(f, p)$, is locally ideal.
\end{prop}
\begin{proof}
The matrix for the constraint system in Eq. \eqref{eq:triangle-chain}, ignoring the constraints defining $x$ and $y$, is a network matrix and hence is totally unimodular. Therefore, all the extreme points of the LP relaxation of Eq. \eqref{eq:triangle-chain} satisfy the integrality restrictions on the vector of  $\bm z$ variables. 
\end{proof}

\begin{cor} \label{cor:sharpness}
Given $f: [\ell, u] \rightarrow \mathbb R$ and an admissible partition $p \in \mathcal A_f[\ell, u]$, $\operatorname{Proj}_{x,y} \mathcal R^{lp}(f, p)$ is the convex hull of $k$ triangles, where $k = |\mathcal P(p)| - 1$. \hfill \qedsymbol
\end{cor}

We remark the only vertices of the triangle that can constitute the extreme points of $\operatorname{Proj}_{x,y} \mathcal R^{lp}(f, p)$ are in the set $\mathcal V = \{\bm v_0, \bm v_k\} \cup \{\bm v_{0, 1}, \dots, \bm v_{k-1, k}\}$ where $k = |\mathcal P(p)| - 1$. This is because the other vertices can be expressed as a convex combination of vertices in $\mathcal V$. In fact, when $f$ is convex, $\mathcal V$ is exactly the set of extreme points of $\operatorname{Proj}_{x,y} \mathcal R^{lp}(f, p)$. Hence, an equivalent linear programming formulation for $\operatorname{Proj}_{x,y} \mathcal R^{lp}(f, p)$ is given by the following convex hull description of vertices in $\mathcal V$:
\begin{flalign}
\operatorname{Proj}_{x,y} \mathcal R^{lp}(f, p) = \operatorname{Proj}_{x,y} \left\{(x, y, \bm \lambda) \in [\ell, u] \times \mathbb R \times \bm \Delta_{|\mathcal V|} : \begin{pmatrix} x \\ y \end{pmatrix} = \sum_{i=1}^{|\mathcal V|} \lambda_i \bm w_i \right\} \label{eq:equivalent-lp}
\end{flalign}
where, $\bm w_1, \bm w_2, \dots, \bm w_{|\mathcal V|}$ are vertices in $\mathcal V$, and $\bm \Delta_{|\mathcal V|}$ is a $|\mathcal V|$-dimensional simplex.

\begin{cor} \label{cor:lp-contains-conv-hull}
Given $f: [\ell, u] \rightarrow \mathbb R$ and an admissible partition $p \in \mathcal A_f[\ell, u]$, $\operatorname{conv}(\mathcal X) \subseteq \operatorname{Proj}_{x,y} \mathcal R^{lp}(f, p) $ is the convex hull of $k$ triangles, where $k = |\mathcal P(p)| - 1$. 
\end{cor}
\begin{proof}
The proof follows from Corollary \ref{cor:sharpness} and the fact that $\mathcal X \subset \mathcal R(f, p)$. 
\end{proof}

\begin{prop} \label{prop:lp-containment}
Given $f:[\ell, u] \rightarrow \mathbb R$ and any infinite sequence of partitions, $\{p_0, p_1, \dots\}$ generated by the algorithm in Sec. \ref{subsec:algo} i.e., using an interval bisection with the bisection occurring in the interval with maximum vertical distance at that iteration, $\operatorname{Proj}_{x,y} \mathcal R^{lp}(f, p_{k+1}) \subset \operatorname{Proj}_{x,y} \mathcal R^{lp}(f, p_{k})$ for any $k \in \mathbb Z_{\geqslant 0}$. 
\end{prop}
\begin{proof}
The proof follows from the observation that for any infinite sequence of partitions $\{p_0, p_1, \dots\}$ generated using the procedure in Sec. \ref{subsec:algo}, $\operatorname{Proj}_{x,y} \mathcal R(f, p_{k+1}) \subset \operatorname{Proj}_{x,y} \mathcal R(f, p_{k})$. This is true because the partition $p_0$ is the base partition of the function $f$ and for any sub-interval $s \in \mathscr S(p_0)$, $f$ is either convex or concave in that sub-interval. This, combined with the fact that the MILP formulation for $\mathcal R(f, p_k)$ is locally ideal for any $k \in \mathbb Z_{\geqslant 0}$, implies $\operatorname{Proj}_{x,y} \mathcal R^{lp}(f, p_{k+1}) \subset \operatorname{Proj}_{x,y} \mathcal R^{lp}(f, p_{k})$.
\end{proof}

Corollary \ref{cor:lp-contains-conv-hull} along with Proposition \ref{prop:lp-containment} basically imply that as algorithm in Sec. \ref{subsec:algo} is applied to any $f: [\ell, u] \rightarrow \mathbb R$, it generates an infinite sequence of partitions $\{p_0, p_1, \dots\}$ such that the projections of the LP relaxations of $\mathcal R(f, p_k)$ to the $(x, y)$ variable space form a decreasing sequence of sets with every projection containing the $\operatorname{conv}(\mathcal X)$. We now present a main theorem that states that this decreasing sequence of sets formed by the projection of the LP relaxations i.e., $\operatorname{Proj}_{x, y} \mathcal R^{lp}(f, p_k)$ indeed converges to $\operatorname{conv}(\mathcal X)$. To do so, we introduce the following definitions and some known results from the literature. 

\begin{definition} \label{def:hausdorff-distance}
Given two closed sets $A$ and $B$, the Hausdorff distance between the sets $A$ and $B$, denoted by $d_H(A, B)$ is defined as 
\begin{flalign}
d_H(A, B) = \max \left\{ \max_{a \in A} \,\min_{b \in B} d(a, b), \max_{b \in B} \, \min_{a \in A} d(a, b)\right\} \label{eq:hausdorff}
\end{flalign}
where, $d(a, b)$ is the Euclidean distance between the points $a$ and $b$.
\end{definition}

\begin{prop} \label{prop:hausdorff-vertex}
If $A$ is a polytope and $B$ is a closed convex subset of $A$, then there exists a vertex $v \in A$ such that $d_H(A, B) = \min_{b\in B} d(v, B)$ i.e., the Hausdorff distance between $A$ and $B$ is achieved at a vertex of $A$.
\end{prop}
\begin{proof}
First we observe that if $B \subset A$, then $d_H(A, B) = \max_{a \in A} \,\min_{b \in B} d(a, b)$. Note that when $B$ is a closed convex set, the function $g(a) = \min_{b \in B} d(a, b)$ is a convex function. Hence, the problem of computing $d_H(A, B)$ reduces to maximizing a convex function over the polytope $A$ and it is known that the maximum for this problem is achieved at a vertex of $A$ \cite{Boyd2004}.
\end{proof}

We now present the main theorem that proves the convergence of the sequence of LP relaxations of the MILPs to $\operatorname{conv}(\mathcal X)$. 

\begin{theorem} \label{thm:lp-convergence}
$\operatorname{Lim}_{k \rightarrow \infty}  \operatorname{Proj}_{x, y} \mathcal R^{lp}(f, p_k)  = \operatorname{conv}(\mathcal X)$. 
\end{theorem}
\begin{proof}
Proving the convergence of the monotonically decreasing sequence of sets $\operatorname{Proj}_{x, y} \mathcal R^{lp}(f, p_k)$ to $\operatorname{conv}(\mathcal X)$ is equivalent to proving the statement that as $k \rightarrow \infty$, $d_H(\operatorname{Proj}_{x, y} \mathcal R^{lp}(f, p_k), \operatorname{conv}(\mathcal X)) \rightarrow 0$ (see \cite{Salinetti1979}). For finite $k$, due to Proposition \ref{prop:hausdorff-vertex}, the Hausdorff distance between $\operatorname{Proj}_{x, y} \mathcal R^{lp}(f, p_k)$ and $\operatorname{conv}(\mathcal X)$, if strictly greater than zero, is achieved at some vertex  of $\operatorname{Proj}_{x, y} \mathcal R^{lp}(f, p_k)$ that is the intersection of two tangents to the function $f(x)$ and this vertex cannot be contained in $\operatorname{conv}(\mathcal X)$. Furthermore, authors in \cite{Burkard1991} prove that the Hausdorff distance between the two sets is bounded above by relaxation strength $\varepsilon(\mathcal R(f, p_k))$ given in Eq. \eqref{eq:relaxation-strength}. This shows that the sequence of Hausdorff distances, being non-negative, is bounded from below by $0$ and is bounded from above by a sequence that converges to $0$ (by Theorem \ref{thm:strength-convergence}). Hence, the sequence of Hausdorff distances converges to zero as well, thereby establishing the set-theoretic convergence of $\operatorname{Lim}_{k \rightarrow \infty}  \operatorname{Proj}_{x, y} \mathcal R^{lp}(f, p_k)$ to the limit $\operatorname{conv}(\mathcal{X})$. 
\end{proof} 

\section{Computational Results} \label{sec:results}
The overarching theme of results presented in this section is to demonstrate the use of our relaxations to solve general mixed-integer nonlinear programs. To that end, instances and univariate functions are hand-picked to show the utility of these relaxations in different contexts. Benchmarking the effectiveness of these relaxations on a wider class of test instances is beyond the scope of this article. All the computational experiments were implemented using the Julia programming language \cite{Bezanson2017} using JuMP v0.21.0 \cite{Dunning2017} and run on a MacBook Pro with a 2.9 GHz Dual-Core Intel Core i5 processor and 16GB RAM. The relaxations presented in this article are also made publicly available as a Julia package (see \url{https://github.com/sujeevraja/PolyhedralRelaxations.jl}).

\paragraph{Strength and size of the MILP relaxations} 
The first set of results is aimed at demonstrating the strength of the MILP relaxation for different univariate functions. For each of the chosen univariate function, we generate MILP relaxations using the interval bisection scheme on the base partition with two stopping criteria (i) the strength of the MILP relaxation (see Eq. \eqref{eq:relaxation-strength}) is less than a pre-specified tolerance, denoted by $\bm \epsilon$, and (ii) the number of partitions excluding the base partition introduced by the refinement scheme does not exceed a pre-specified number, denoted by $\bm b$. We report the number of partitions introduced in the relaxation and the strength of the relaxation obtained using the stopping criteria (i) and (ii), respectively. Table \ref{tab:basic} presents these results for four nonlinear non-convex univariate functions. The functions $\sin x$, $x\cdot |x|$ and $S(x)$ were chosen because these univariate functions frequently occur in optimization problems concerning power networks \cite{Lu2018}, natural gas networks \cite{Sundar2018} and logistic regression \cite{Lleinbaum2002}, respectively. Note that in Table \ref{tab:basic}, the lower the value of $\bm \epsilon$, the tighter the relaxation.
\begingroup
\setlength{\tabcolsep}{10pt} 
\renewcommand{\arraystretch}{1.5} 
\begin{table}[ht]
    \footnotesize
    \centering
    \begin{tabular}{ccccccccc}
        \toprule
        \multirow{2}{*}{$f(x)$} & \multirow{2}{*}{$[\ell, u]$} & \multirow{2}{*}{$p_f^0$} & \multicolumn{3}{c}{\# partitions} & \multicolumn{3}{c}{strength of relaxation} \\ 
        \cmidrule{4-9}
        & & & $\bm \epsilon = \infty$ & $\bm \epsilon = 0.1$ & $\bm \epsilon = 0.01$ & $\bm b = 0$ & $\bm b = 50$ & $\bm b = 100$ \\ 
        \midrule
        $\sin x$ & $[0, 2\pi]$ & $[0, \pi, 2\pi]$ & 2 & 12 & 28 & 1.5707 & 0.0009 & 0.0009 \\
        $x^3$ & $[-1, 1]$ & $[-1, 0, 1]$ & 2 & 6 & 26 & 0.7500 & 0.0014 & 0.0004 \\
        $x\cdot |x|$ & $[-2, 2]$ & $[-2, 0, 2]$ & 2 & 16 & 32 & 2.0000 & 0.0078 & 0.0020 \\
        $S(x)$ & $[-5, 5]$ & $[-5, 0, 5]$ & 2 & 6 & 14 & 0.3042 & 0.0009 & 4.32E-5 \\
        \bottomrule
    \end{tabular}
    \caption{Here, $S(x)$ is the logistic function defined as $1/{(1+e^{-x})}$. The number of partitions corresponding to $\bm \epsilon = \infty$ is the number of partitions in $p_f^0$ and the strength of the relaxation for $\bm b=0$ is $\varepsilon(\mathcal R(f, p_f^0))$.}
    \label{tab:basic}
\end{table}

\paragraph{Visualization of the LP relaxations} 
Here, we show how the LP relaxations evolve with increasing number of partitions to illustrate the convergence of the sequence of LP relaxations to the convex hull of the graph of the univariate function. To that end, we choose three univariate functions $\sin x$, $x \cdot |x|$, and $x^4 - x^3$. Fig. \ref{fig:envelopes-sin}--\ref{fig:envelopes-poly} show the projection of the LP relaxation in the space of $(x, y)$ variables with increasing partition count for these functions. In all figures, $\bm b$ is the number of partitions excluding the base partition of the univariate functions considered. The red curve denotes the actual function; the green and the blue lines are the over- and under-estimators, respectively.

\begin{figure}[!ht]
\centering
    \begin{subfigure}[t]{0.3\textwidth}
        \centering
        \includegraphics[height=1.2in]{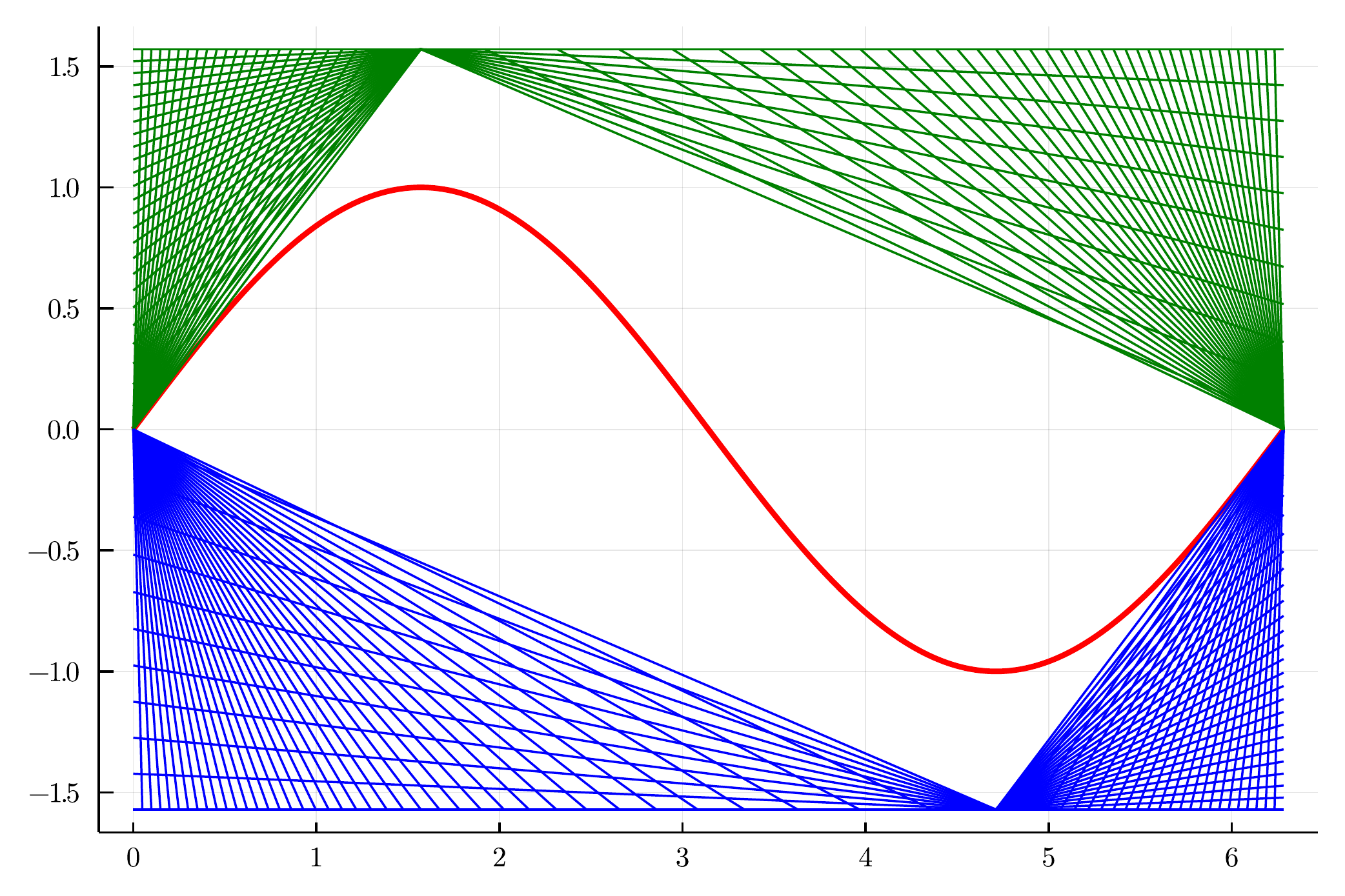}
        \caption{$\bm b = 0$}
    \end{subfigure}%
    \centering
    \begin{subfigure}[t]{0.3\textwidth}
        \centering
        \includegraphics[height=1.2in]{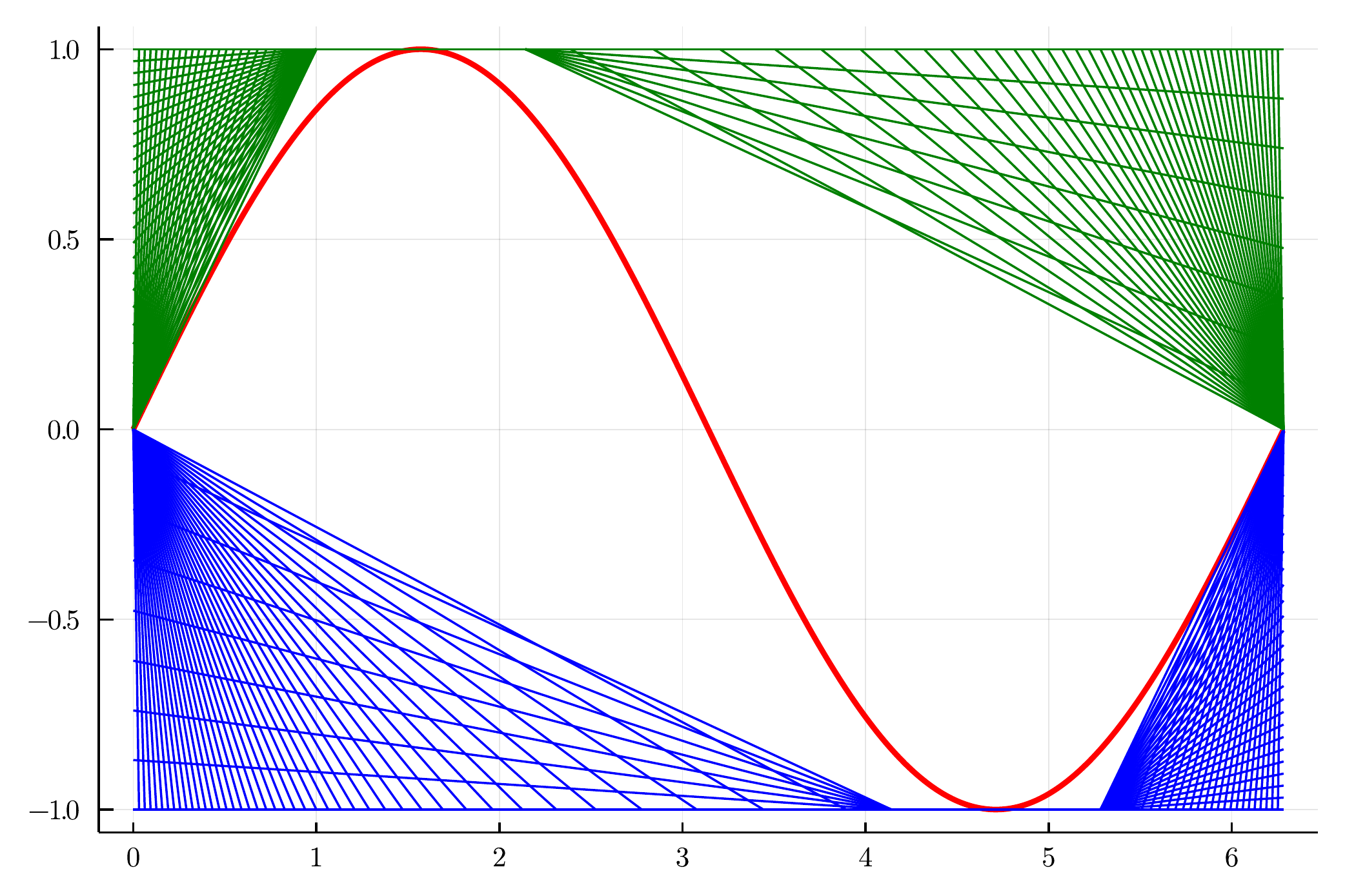}
        \caption{$\bm b = 2$}
    \end{subfigure}%
    \centering
    \begin{subfigure}[t]{0.3\textwidth}
        \centering
        \includegraphics[height=1.2in]{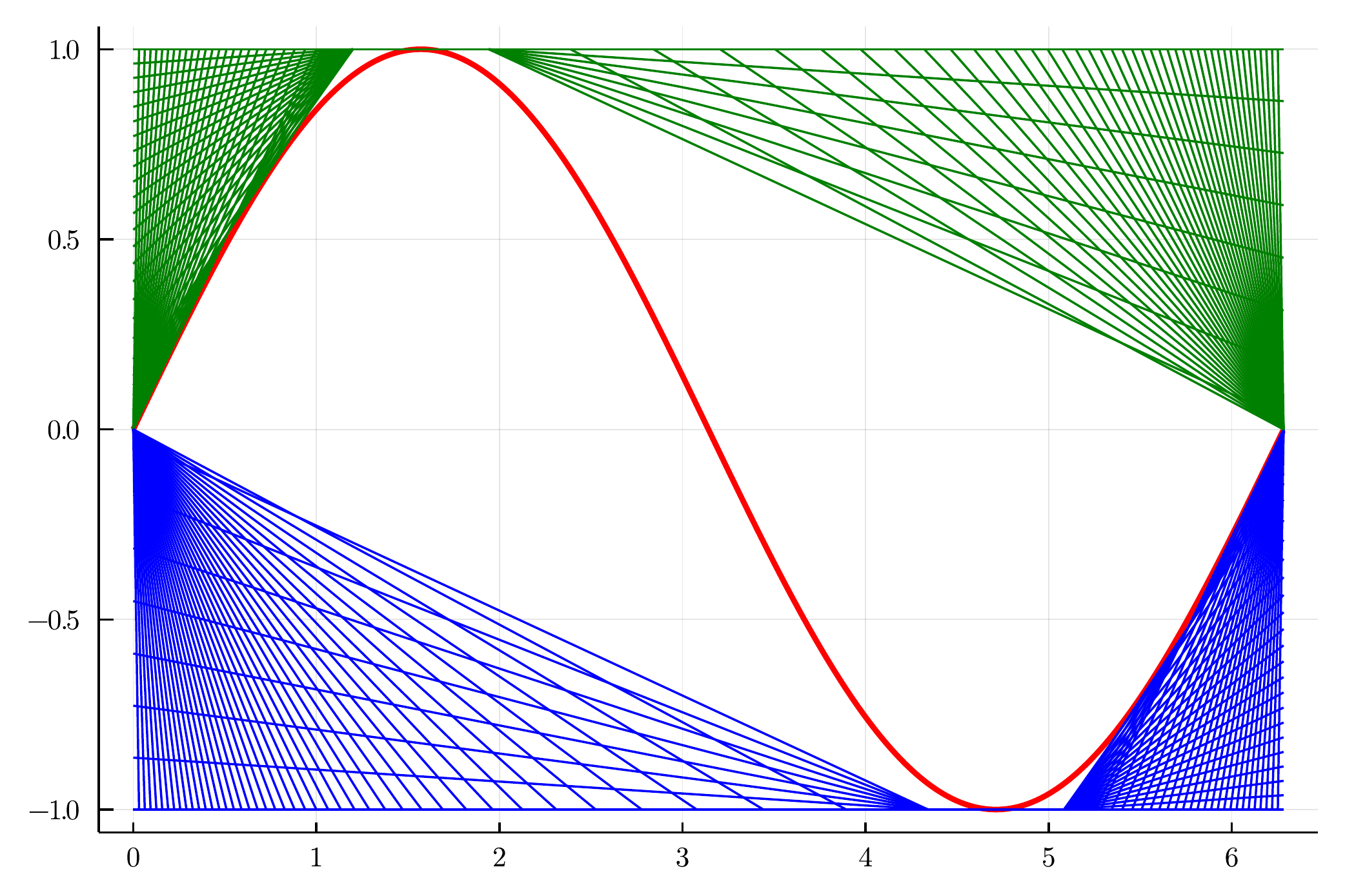}
        \caption{$\bm b = 6$}
    \end{subfigure}%
    \caption{LP relaxations of the convex hull of $y = \sin x$ with a domain of $[0, 2\pi]$.}
    \label{fig:envelopes-sin}
\end{figure}
\begin{figure}[!ht]
\centering
    \begin{subfigure}[t]{0.3\textwidth}
        \centering
        \includegraphics[height=1.2in]{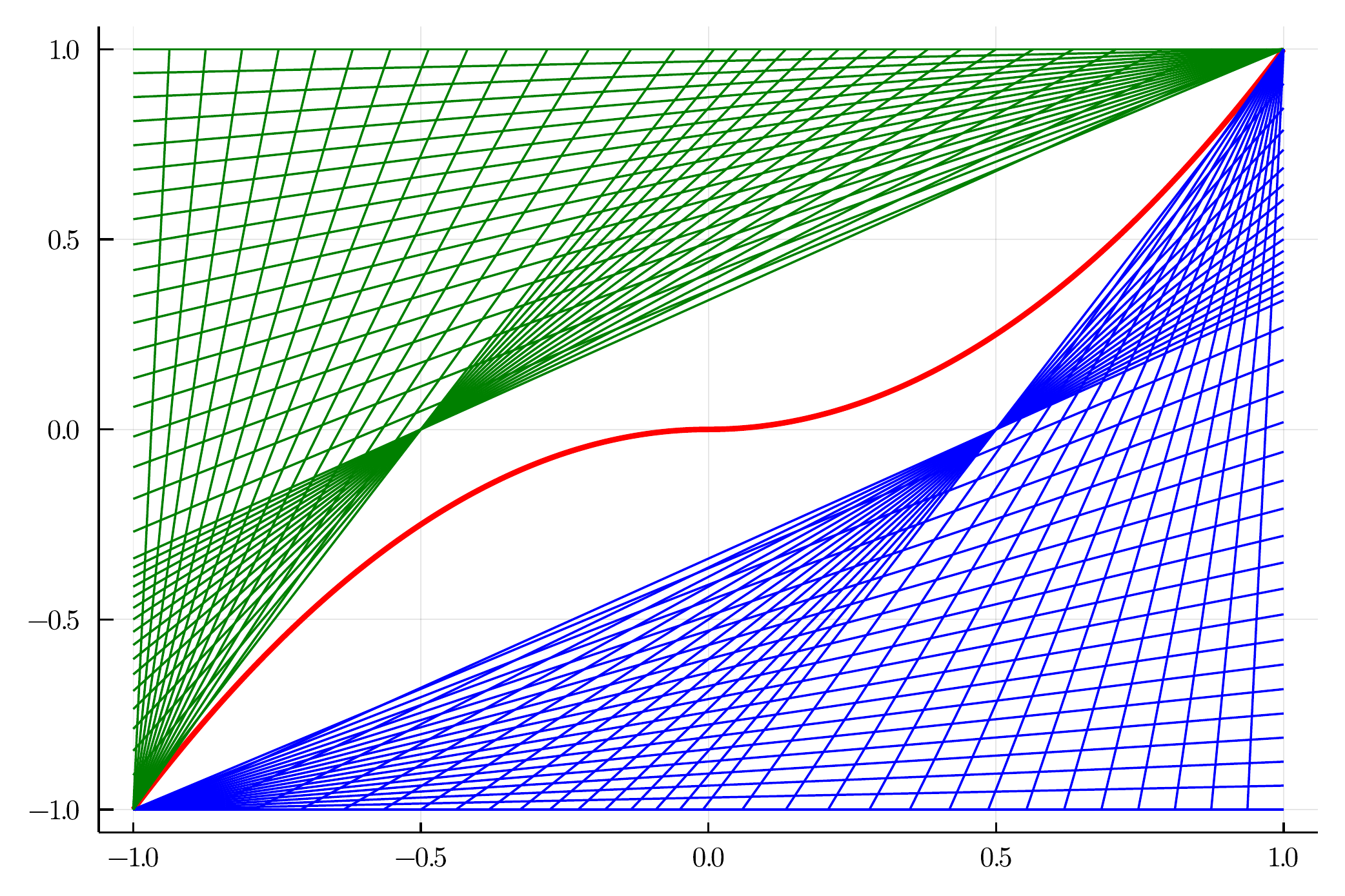}
        \caption{$\bm b = 0$}
    \end{subfigure}%
    \centering
    \begin{subfigure}[t]{0.3\textwidth}
        \centering
        \includegraphics[height=1.2in]{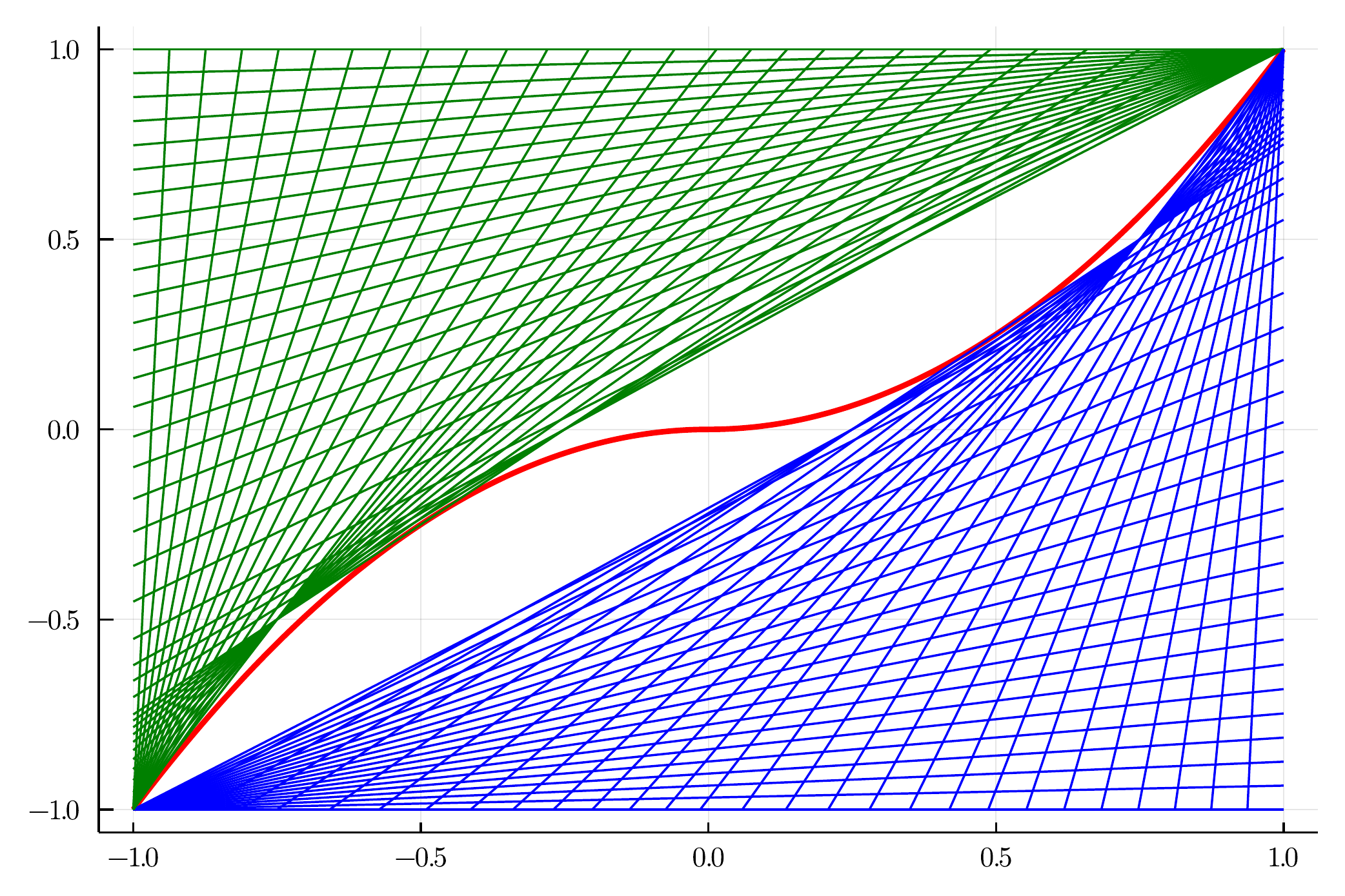}
        \caption{$\bm b = 2$}
    \end{subfigure}%
    \centering
    \begin{subfigure}[t]{0.3\textwidth}
        \centering
        \includegraphics[height=1.2in]{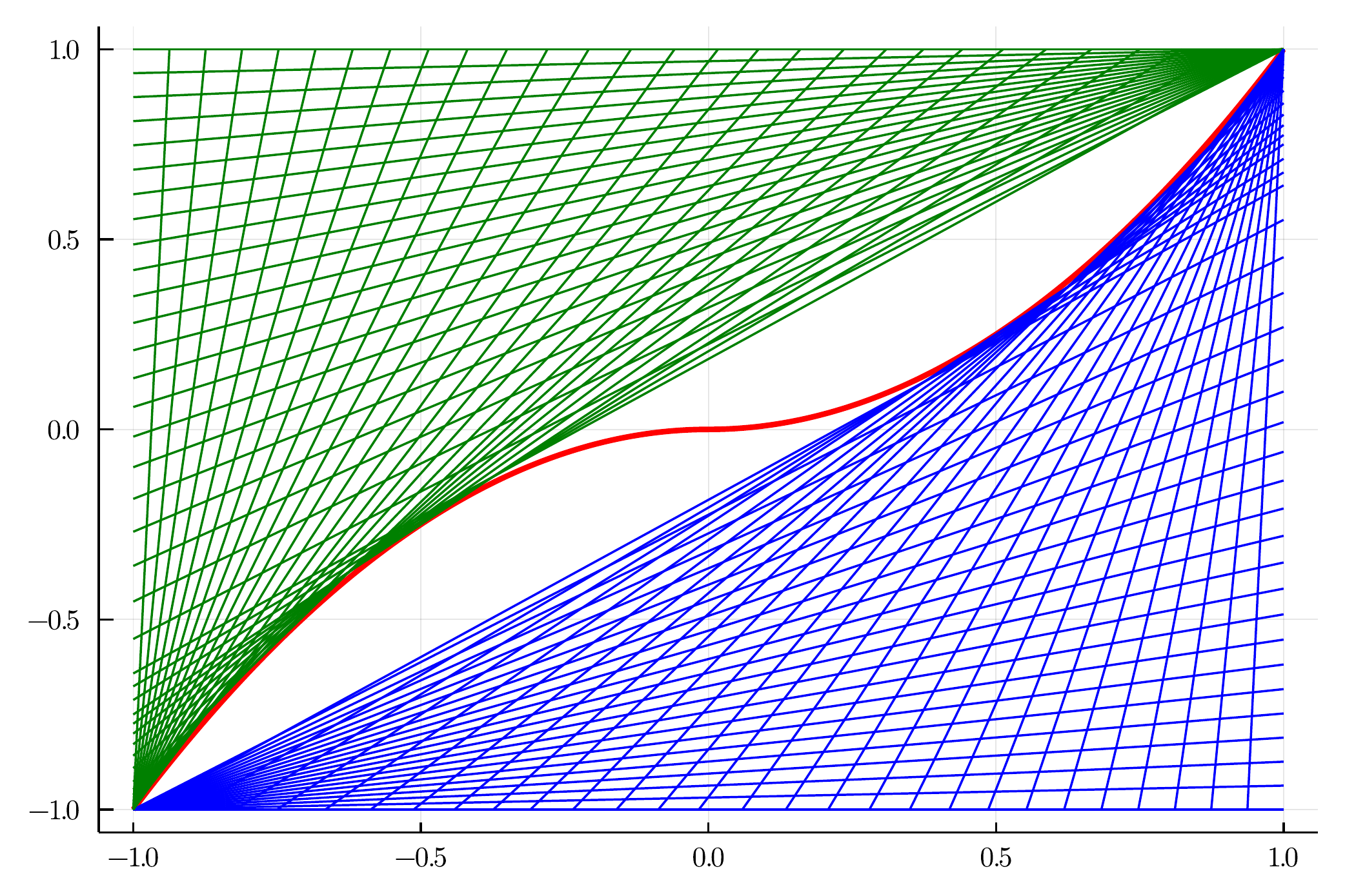}
        \caption{$\bm b = 6$}
    \end{subfigure}%
    \caption{LP relaxations of the convex hull of $y = x\cdot |x|$ with a domain of $[-1, 1]$.}
    \label{fig:envelopes-xabsx}
\end{figure}
\begin{figure}[!ht]
\centering
    \begin{subfigure}[t]{0.3\textwidth}
        \centering
        \includegraphics[height=1.2in]{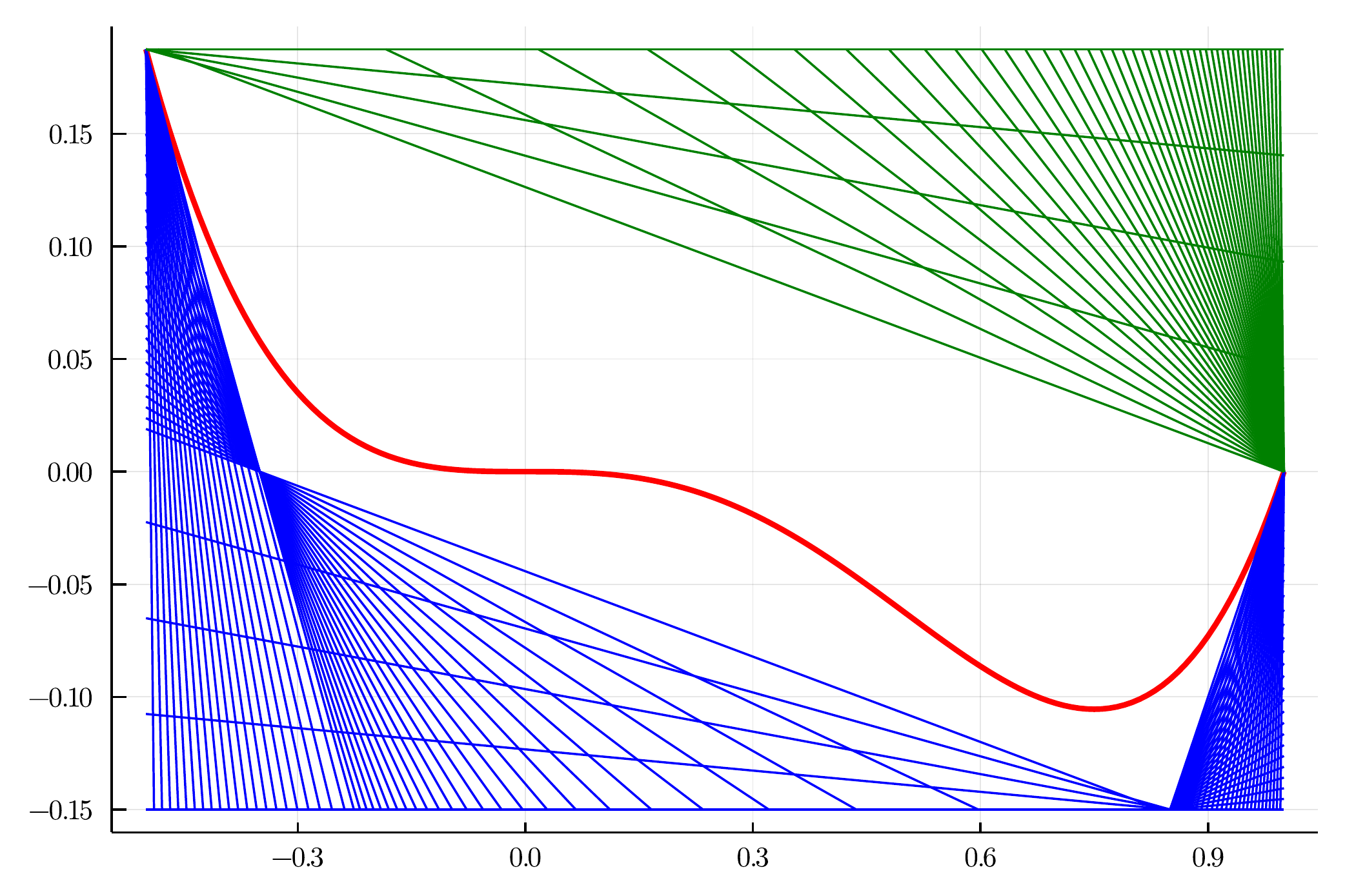}
        \caption{$\bm b = 0$}
    \end{subfigure}%
    \centering
    \begin{subfigure}[t]{0.3\textwidth}
        \centering
        \includegraphics[height=1.2in]{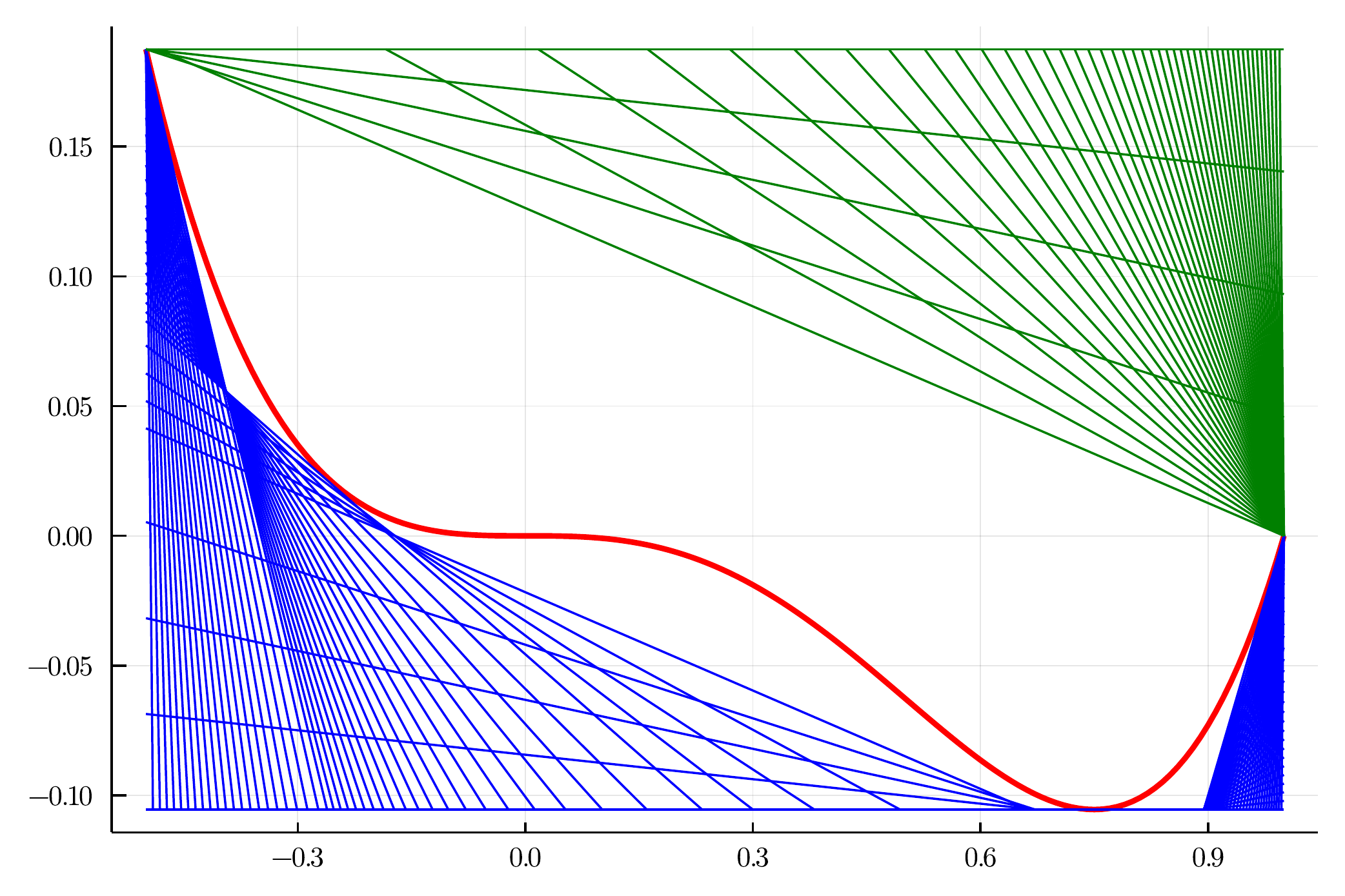}
        \caption{$\bm b = 3$}
    \end{subfigure}%
    \centering
    \begin{subfigure}[t]{0.3\textwidth}
        \centering
        \includegraphics[height=1.2in]{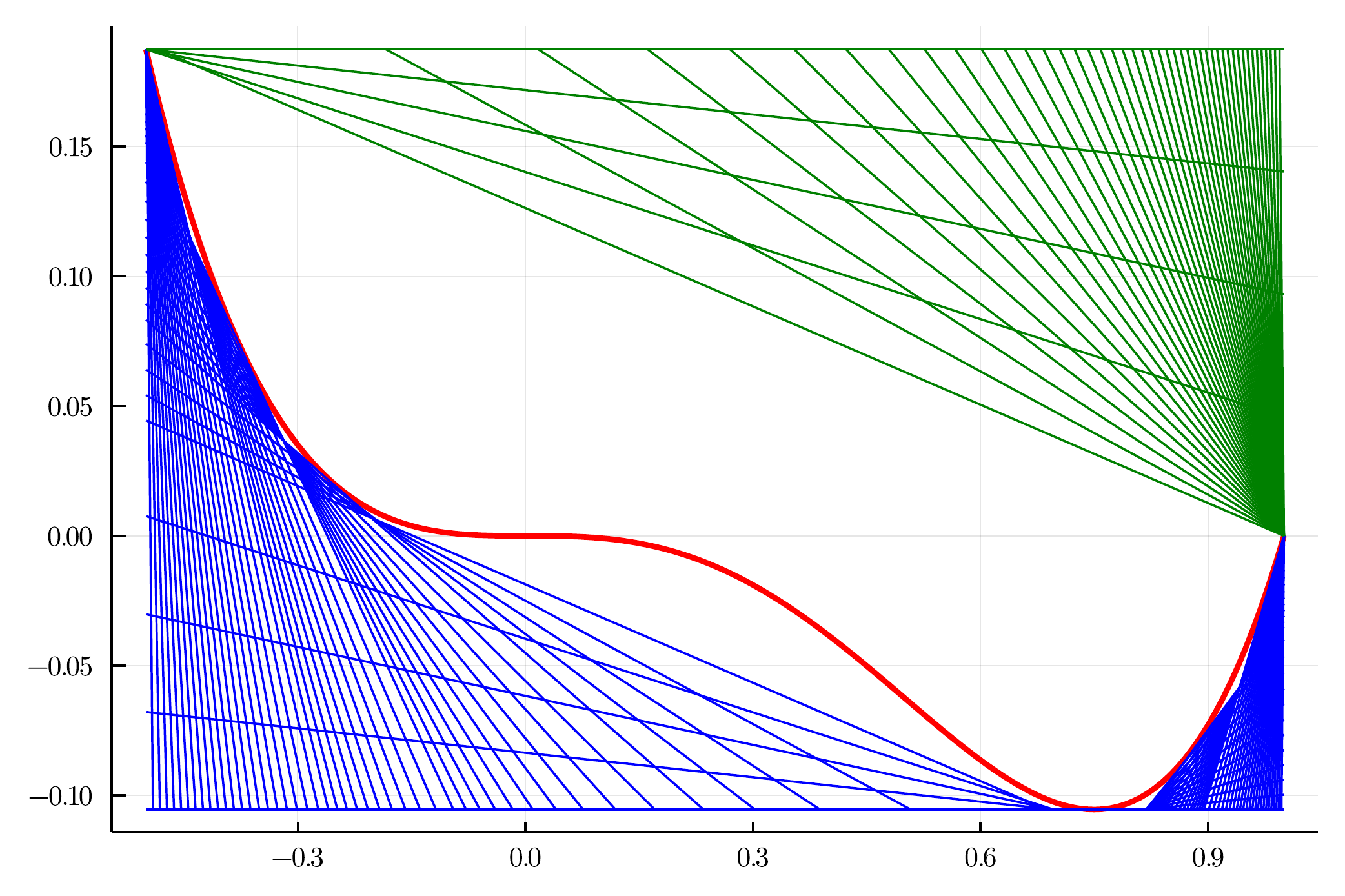}
        \caption{$\bm b = 9$}
    \end{subfigure}%
    \caption{LP relaxations of the convex hull of $y = x^4 - x^3$ with a domain of $[-0.5, 1.0]$.}
    \label{fig:envelopes-poly}
\end{figure}

\paragraph{Relaxations for separable Mixed-Integer Nonlinear Programs from MINLPLib}
The results presented in Tables \ref{tab:milp} and \ref{tab:lp} aim to show the strength of the MILP and the LP relaxations on separable Mixed-Integer Nonlinear programs/Nonlinear Programs (MINLP/NLP) instances from MINLPLib \cite{Bussieck2003}. MINLPs/NLPs are said to be separable if each nonlinear function in the problem can be decomposed into a sum of univariate functions. Four separable instances (i) \texttt{trig} (ii) \texttt{ex4\_1\_1}, and (iii) \texttt{fo7} (iv) \texttt{ramsey} are chosen to illustrate the effectiveness of relaxations presented in the article.

The \texttt{trig}, \texttt{ramsey} and \texttt{fo7} instances are chosen to illustrate the relaxations' effectiveness on a separable NLPs with trigonometric functions, $\log$, $x^a$ where $a$ is fractional and MINLP respectively. The instance \texttt{ex4\_1\_1} is chosen to compare the quality of relaxations on different reformulations of the problem. We present the original NLP of \texttt{ex4\_1\_1} (Eq. \eqref{eq:full}) and its two equivalent reformulations (Eq. \eqref{eq:reformulation1} and Eq. \eqref{eq:reformulation2}) below:
\begin{flalign}
\texttt{ex4\_1\_1}: \quad \min_{x \in [-2, 11]} \quad 0.1 + x^6 - 2.08 x^5 + 0.4875 x^4 + 7.1 x^3 - 3.95  x^2 - x 
\label{eq:full}
\end{flalign}
\begin{subequations}
\begin{flalign}
\texttt{ex4\_1\_1}^a: \quad & \min_{x \in [-2, 11]} \quad y \label{eq:r1-obj} \\ 
\text{ subject to :} \quad & y = 0.1 + x^6 - 2.08 x^5 + 0.4875 x^4 + 7.1 x^3 - 3.95  x^2 - x \label{eq:r1-lift} 
\end{flalign}
\label{eq:reformulation1}
\end{subequations}
\begin{subequations}
\begin{flalign}
\texttt{ex4\_1\_1}^b: \quad & \min_{x \in [-2, 11]}\quad 0.1 + y_1 + y_2 + y_3 + y_4 + y_5 - x  \label{eq:r2-obj} \\ 
\text{ subject to :} \quad & y_1 = x^6 \quad  y_2 = - 2.08 x^5 \quad y_3 =  0.4875 x^4  \label{eq:r2-lift1} \\ 
y_4 = 7.1 x^3 \quad y_5 = -3.95  x^2  \label{eq:r2-lift2} 
\end{flalign}
\label{eq:reformulation2}
\end{subequations}
The reformulations differ in the way univariate functions are lifted by addition of auxiliary variables.
When the MILP and the LP relaxations are applied to the two reformulations \texttt{ex4\_1\_1}$^a$ and \texttt{ex4\_1\_1}$^b$, the strength of the respective relaxations differ in the sense that the reformulation \texttt{ex4\_1\_1}$^a$ always produces better solutions than the reformulation \texttt{ex4\_1\_1}$^b$. On the other hand, some computational effort is required to obtain the base partition of the univariate function in Eq. \eqref{eq:r1-lift} unlike the five univariate functions in Eq. \eqref{eq:r2-lift1}--\eqref{eq:r2-lift2}. The takeaway from this exercise is that it is better to find base partitions of a single univariate function that is expressed as a sum/difference of univariate terms, instead of breaking them up into multiple univariate functions and using partitions of the functions. In Tables \ref{tab:milp} and \ref{tab:lp}, the relative gap \% is the gap between the lower bound (LB) and the global optimal solution in the column ``g-opt''. Results from Table \ref{tab:milp} show that just the application of the vanilla MILP relaxation, without additional enhancements, constructed using the algorithm in Sec. \ref{subsec:algo} yields a solution within 1\% of the global optimum for the chosen instances. The results in Table \ref{tab:lp} indicate that quick LP-based lower bounds can also be obtained by utilizing the formulation in Eq. \eqref{eq:equivalent-lp} to relax the graph of the nonlinear univariate function. 

\begin{table}[ht]
    \footnotesize
    \centering
    \begin{tabular}{cccccccc}
        \toprule
        \multirow{2}{*}{instance} & \multirow{2}{*}{g-opt} & \multicolumn{3}{c}{LB (relative gap \%)} & \multicolumn{3}{c}{computation time (sec)} \\ 
        \cmidrule{3-8}
        & & $\bm \epsilon = \infty$ & $\bm \epsilon = 0.1$ & $\bm \epsilon = 0.01$ & $\bm \epsilon = \infty$ & $\bm \epsilon = 0.1$ & $\bm \epsilon = 0.01$ \\
        \midrule
        \texttt{trig} & -3.7625 & -5.0749 (25.86) & -3.7943 (0.84) & -3.7694 (0.18) & 1.97 & 1.02 & 3.55 \\
        \texttt{ramsey} & -2.4874 & -3.0637 (18.80) & -3.0637 (18.80) & -2.5309 (1.7) & 0.0 & 0.0 & 0.01 \\
        \texttt{ex4\_1\_1}$^a$ & -7.4873 & -15.8046 (52.63) & -7.5239 (0.48) & -7.4892 (0.02) & 0.07 & 1.14 & 3.49  \\
        \texttt{ex4\_1\_1}$^b$ & -7.4873 & -27.24E4 (99.99) & -7.5633 (1.00) & -7.4941 (0.09) & 0.01 & 25.23 & 161.23 \\
        \texttt{fo7} & 20.7298 & 16.7051 (24.09) & 18.7674 (10.46) & 20.5481 (0.88) & 3.0 & 9.15 & 50.79 \\
        \bottomrule
    \end{tabular}
    \caption{MILP relaxation values on MINLPLib instances}
    \label{tab:milp}
\end{table}

\begin{table}[ht]
    \footnotesize
    \centering
    \begin{tabular}{cccccccc}
        \toprule
        \multirow{2}{*}{instance} & \multirow{2}{*}{g-opt} & \multicolumn{3}{c}{LB (relative gap \%)} & \multicolumn{3}{c}{computation time (sec)} \\ 
        \cmidrule{3-8}
        & & $\bm \epsilon = \infty$ & $\bm \epsilon = 0.1$ & $\bm \epsilon = 0.01$ & $\bm \epsilon = \infty$ & $\bm \epsilon = 0.1$ & $\bm \epsilon = 0.01$ \\
        \midrule
        \texttt{trig} & -3.7625 & -6.2831 (40.11) & -4.0377 (6.81) & -4.0034 (6.01) & 0.00 & 0.01 & 0.01\\
        \texttt{ramsey} & -2.4874 & -3.0637 (18.80) & -3.0637 (18.80) & -2.5305 (1.7) & 0.0 & 0.0 & 0.0 \\
        \texttt{ex4\_1\_1}$^a$ & -7.4873 & -15.8046 (52.62) & -7.5239 (0.48) & -7.4892 (0.02) & 0.0 & 0.03 & 0.13 \\
        \texttt{ex4\_1\_1}$^b$ & -7.4873 & -28.10E4 (99.99) & -16.47E4 (99.99) & -16.47E4 (99.99) & 0.00 & 0.21 & 0.30  \\
        \texttt{fo7} & 20.7298 & 16.7051 (24.09) & 18.7674 (10.46) & 20.5481 (0.88) & 6.44 & 10.18 & 14.85 \\
        \bottomrule
    \end{tabular}
    \caption{Relaxation objective values obtained on MINLPLib instances when the LP relaxation is applied to the univariate functions in the problem.}
    \label{tab:lp}
\end{table}

\paragraph{Relaxations for $\operatorname{erf}(x)$ and $\Gamma(x)$} This set of results is aimed at demonstrating the efficacy of the proposed relaxations for NLPs with the gamma function, $\Gamma(x)$, and the Gaussian error function, $\operatorname{erf}(x)$. Both these functions are only defined using integrals as follows:
\begin{flalign}
\Gamma(x) &= \int_0^{\infty} t^{x-1} \exp{(-t)} ~dt, \quad x > 0, \label{eq:gamma} \\
\operatorname{erf}(x) &= \frac{2}{\sqrt{\pi}} \int_0^x \exp{(-t^2)} ~ dt. \label{eq:erf}
\end{flalign}
Here, we demonstrate another facet of the proposed relaxations, namely the reliance only on oracles that provide values of the function and its derivative. In other words, we do not require closed-form expressions of a function or its derivative to formulate its relaxations. Both functions considered in this set of results lack closed form expressions. For the gamma function, even the derivative does not have a closed form expression. Nevertheless, there do exist oracles that can provide values of these functions and their derivatives up to an arbitrary tolerance. To demonstrate the effectiveness of the relaxations on NLPs with the above functions, we consider the following two optimization problems:
\begin{subequations}
\begin{flalign}
\texttt{erf}: &\quad \min\quad  \operatorname{erf}(x+y) + \sin x \cdot \exp{\left(\frac{-y}2\right)} \\ 
\text{subject to}: & \quad x^2 + y^2 \geqslant 10, 0\leqslant x, y \leqslant 10.
\end{flalign}
\label{eq:errf}
\end{subequations}
\begin{flalign}
\texttt{gamma}: & \quad \min \quad \Gamma(x) \quad \text{subject to: } 0.5 \leqslant x \leqslant 5. \label{eq:gammaf}
\end{flalign}
The globally optimal solutions for the NLPs in Eq. \eqref{eq:errf} and \eqref{eq:gammaf} are $0.0$ \cite{Mccormick1985} and $0.8856$ \cite{Gamma}, respectively. These optimization problems cannot be directly provided to state-of-the-art global optimization solvers like BARON or SCIP. In the NLP in Eq. \eqref{eq:errf}, we note that the objective contains a bilinear function. We utilize the transformation $w_1 = \sin x + \exp{\left(-y/2\right)}$, $w_2 = \sin x-\exp{\left(-y/2\right)}$ to transform $\sin x \cdot \exp{\left(-y/2\right)}$ to $(w_1^2 -w_2^2)/4$. This enables direct use of the formulation proposed in this paper to compute tight lower bounds. The strength of the MILP and the LP relaxations for these two NLPs are shown in Table \ref{tab:custom}. For this set of results, we do not report computation times since all runs were completed within a fraction of a second. In Table \ref{tab:custom}, MILP-LB and the LP-LB are values of the lower bound obtained by solving the piecewise polyhedral relaxations and the polyhedral relaxations, respectively, for the corresponding value of $\bm \epsilon$. 
\begin{table}[ht]
    \footnotesize
    \centering
    \begin{tabular}{cccccc}
        \toprule
        \multirow{2}{*}{instance} & \multirow{2}{*}{g-opt} & \multicolumn{2}{c}{MILP-LB} & \multicolumn{2}{c}{LP-LB} \\ 
        \cmidrule{3-4} \cmidrule{5-6}
        & & $\bm \epsilon = \infty$ & $\bm \epsilon = 0.001$ & $\bm \epsilon = \infty$ & $\bm \epsilon = 0.001$ \\
        \midrule
        \texttt{erf} & 0.0 & -0.5739 & 0.0 & -0.5739 & -0.3692 \\
        \texttt{gamma} & 0.8856 & -10.5610 & 0.8855 & -10.5610 & 0.8855\\
        \bottomrule
    \end{tabular}
    \caption{MILP and LP relaxation objective values obtained for the NLPs in Eqs. \eqref{eq:errf} and \eqref{eq:gammaf}.}
    \label{tab:custom}
\end{table}
As mentioned in the introduction, the relaxations for functions like the error function and the gamma function (shown to be very effective in Table \ref{tab:custom}) can directly be integrated into either sB\&B-based or MILP-based global optimization algorithms resulting in global solution techniques to MINLPs with a wider class of non-convex structures.

\section{Conclusion} \label{sec:conclusion}
We present a systematic way to construct a sequence of MILP and LP relaxations for the graph of a univariate function in $\mathcal C^1[\ell, u]$ i.e., bounded, continuous, and differentiable with domain $[\ell, u]$. The sequence of MILP and LP relaxations is proved to converge to the graph of the function and to its convex hull when the domain of the function is partitioned with an interval bisection scheme. To the best of our knowledge, this is the first work in the literature that shows such convergence results for both MILP and the LP relaxations in a single framework. Future work would focus on studying the efficacy of these relaxations when put to use in both sB\&B-based and  MILP-based global optimization algorithms for general non-convex optimization problems.

\section*{Acknowledgements}
The work was funded LANL’s Directed Research and Development (LDRD) projects, ``20170201ER: POD: A Polyhedral Outer-approximation, Dynamic-discretization optimization solver" and ``20200603ECR: Distributed Algorithms for Large-Scale Ordinary Differential/Partial Differential Equation (ODE/PDE) Constrained Optimization Problems on Graphs". This work was carried out under the U.S. DOE Contract No. DE-AC52-06NA25396.

\bibliographystyle{plain}
\bibliography{references}

\begin{thebibliography}{10}

\bibitem{Bezanson2017}
Jeff Bezanson, Alan Edelman, Stefan Karpinski, and Viral~B Shah.
\newblock Julia: A fresh approach to numerical computing.
\newblock {\em {SIAM} review}, 59(1):65--98, 2017.

\bibitem{Boyd2004}
Stephen Boyd and Lieven Vandenberghe.
\newblock {\em Convex optimization}.
\newblock Cambridge university press, 2004.

\bibitem{Burkard1991}
Rainer~E Burkard, Horst~W Hamacher, and G{\"u}nter Rote.
\newblock Sandwich approximation of univariate convex functions with an
  application to separable convex programming.
\newblock {\em Naval Research Logistics (NRL)}, 38(6):911--924, 1991.

\bibitem{Bussieck2003}
Michael~R Bussieck, Arne~Stolbjerg Drud, and Alexander Meeraus.
\newblock {MINLPLib}—a collection of test models for mixed-integer nonlinear
  programming.
\newblock {\em INFORMS Journal on Computing}, 15(1):114--119, 2003.

\bibitem{Castillo2018}
Pedro A~Castillo Castillo, Pedro~M Castro, and Vladimir Mahalec.
\newblock Global optimization of {MIQCP}s with dynamic piecewise relaxations.
\newblock {\em Journal of Global Optimization}, 71(4):691--716, 2018.

\bibitem{Castro2016}
Pedro~M Castro.
\newblock Normalized multiparametric disaggregation: an efficient relaxation
  for mixed-integer bilinear problems.
\newblock {\em Journal of Global Optimization}, 64(4):765--784, 2016.

\bibitem{Dambrosio2012}
Claudia D'Ambrosio, Jon Lee, and Andreas W{\"a}chter.
\newblock An algorithmic framework for {MINLP} with separable non-convexity.
\newblock In {\em Mixed Integer Nonlinear Programming}, pages 315--347.
  Springer, 2012.

\bibitem{Gamma}
W~Edwards Deming and Clarence~G Colcord.
\newblock The minimum in the gamma function.
\newblock {\em Nature}, 135(3422):917--917, 1935.

\bibitem{Dunning2017}
Iain Dunning, Joey Huchette, and Miles Lubin.
\newblock {JuMP: A Modeling Language for Mathematical Optimization}.
\newblock {\em {SIAM} Review}, 59(2):295--320, 2017.

\bibitem{Floudas2014}
Christodoulos~A Floudas and Panos~M Pardalos.
\newblock {\em Recent advances in global optimization}.
\newblock Princeton University Press, 2014.

\bibitem{Gounaris2008}
Chrysanthos~E Gounaris and Christodoulos~A Floudas.
\newblock Tight convex underestimators for $\mathcal{C}^2$-continuous problems:
  I. univariate functions.
\newblock {\em Journal of Global Optimization}, 42(1):51--67, 2008.

\bibitem{Jeroslow1984}
Robert~G Jeroslow and James~K Lowe.
\newblock Modelling with integer variables.
\newblock In {\em Mathematical Programming at Oberwolfach II}, pages 167--184.
  Springer, 1984.

\bibitem{Kallrath2014}
Josef Kallrath and Steffen Rebennack.
\newblock Computing area-tight piecewise linear overestimators, underestimators
  and tubes for univariate functions.
\newblock In {\em Optimization in Science and Engineering}, pages 273--292.
  Springer, 2014.

\bibitem{Lleinbaum2002}
David~G Kleinbaum, K~Dietz, M~Gail, Mitchel Klein, and Mitchell Klein.
\newblock {\em Logistic regression}.
\newblock Springer, 2002.

\bibitem{Liberti2003}
Leo Liberti and Constantinos~C Pantelides.
\newblock Convex envelopes of monomials of odd degree.
\newblock {\em Journal of Global Optimization}, 25(2):157--168, 2003.

\bibitem{Lu2018}
Mowen Lu, Harsha Nagarajan, Russell Bent, Sandra~D Eksioglu, and Scott~J Mason.
\newblock Tight piecewise convex relaxations for global optimization of optimal
  power flow.
\newblock In {\em Power Systems Computation Conference (PSCC)}, pages 1--7.
  IEEE, 2018.

\bibitem{Mccormick1985}
Garth~P McCormick.
\newblock Global solutions to factorable nonlinear optimization problems using
  separable programming techniques.
\newblock 1985.

\bibitem{Misener2011}
Ruth Misener, Jeffrey~P Thompson, and Christodoulos~A Floudas.
\newblock {APOGEE}: {G}lobal optimization of standard, generalized, and
  extended pooling problems via linear and logarithmic partitioning schemes.
\newblock {\em Computers \& Chemical Engineering}, 35(5):876--892, 2011.

\bibitem{Nagarajan2019}
Harsha Nagarajan, Mowen Lu, Site Wang, Russell Bent, and Kaarthik Sundar.
\newblock An adaptive, multivariate partitioning algorithm for global
  optimization of nonconvex programs.
\newblock {\em Journal of Global Optimization}, 74(4):639--675, 2019.

\bibitem{Nagarajan2016}
Harsha Nagarajan, Mowen Lu, Emre Yamangil, and Russell Bent.
\newblock Tightening {McCormick} relaxations for nonlinear programs via dynamic
  multivariate partitioning.
\newblock In {\em International Conference on Principles and Practice of
  Constraint Programming}, pages 369--387. Springer, 2016.

\bibitem{GurobiBenchmarks}
Gurobi Optimization.
\newblock Gurobi 8 performance benchmarks, 2019.
\newblock \url{https://www.gurobi.com/pdfs/benchmarks.pdf}.

\bibitem{Padberg2012}
Manfred~W Padberg and Minendra~P Rijal.
\newblock {\em Location, scheduling, design and integer programming}, volume~3.
\newblock Springer Science \& Business Media, 2012.

\bibitem{Pardalos2013}
Panos~M Pardalos and H~Edwin Romeijn.
\newblock {\em Handbook of global optimization}, volume~2.
\newblock Springer Science \& Business Media, 2013.

\bibitem{Rebennack2016}
Steffen Rebennack.
\newblock Computing tight bounds via piecewise linear functions through the
  example of circle cutting problems.
\newblock {\em Mathematical Methods of Operations Research}, 84(1):3--57, 2016.

\bibitem{baron}
Nikolaos~V Sahinidis.
\newblock Baron: A general purpose global optimization software package.
\newblock {\em Journal of global optimization}, 8(2):201--205, 1996.

\bibitem{Salinetti1979}
Gabriella Salinetti and Roger J-B Wets.
\newblock On the convergence of sequences of convex sets in finite dimensions.
\newblock {\em SIAM review}, 21(1):18--33, 1979.

\bibitem{Smith1997}
Edward~MB Smith and Constantinos~C Pantelides.
\newblock Global optimisation of nonconvex {MINLP}s.
\newblock {\em Computers \& Chemical Engineering}, 21:S791--S796, 1997.

\bibitem{Sundar2020}
Kaarthik Sundar, Harsha Nagarajan, Jeff Linderoth, Site Wang, and Russell Bent.
\newblock Piecewise polyhedral formulations for a multilinear term.
\newblock {\em Operations Research Letters}, 49(1):144--149, 2021.

\bibitem{Sundar2018}
Kaarthik Sundar and Anatoly Zlotnik.
\newblock State and parameter estimation for natural gas pipeline networks
  using transient state data.
\newblock {\em IEEE Transactions on Control Systems Technology},
  27(5):2110--2124, 2018.

\bibitem{Tawarmalani2002}
Mohit Tawarmalani, Nikolaos~V Sahinidis, and Nikolaos Sahinidis.
\newblock {\em Convexification and global optimization in continuous and
  mixed-integer nonlinear programming: theory, algorithms, software, and
  applications}, volume~65.
\newblock Springer Science \& Business Media, 2002.

\bibitem{Teles2013}
Jo{\~a}o~P Teles, Pedro~M Castro, and Henrique~A Matos.
\newblock Univariate parameterization for global optimization of mixed-integer
  polynomial problems.
\newblock {\em European Journal of Operational Research}, 229(3):613--625,
  2013.

\bibitem{Thakur1980}
Lakshman~S Thakur.
\newblock Error analysis for convex separable programs: {B}ounds on optimal and
  dual optimal solutions.
\newblock {\em Journal of Mathematical Analysis and Applications},
  75(2):486--494, 1980.

\bibitem{Wicaksono2008}
Danan~Suryo Wicaksono and Iftekhar~A Karimi.
\newblock Piecewise {MILP} under-and overestimators for global optimization of
  bilinear programs.
\newblock {\em AIChE Journal}, 54(4):991--1008, 2008.

\bibitem{Yildiz2013}
Sercan Y{\i}ld{\i}z and Juan~Pablo Vielma.
\newblock Incremental and encoding formulations for mixed integer programming.
\newblock {\em Operations Research Letters}, 41(6):654--658, 2013.

\end{thebibliography}

\end{document}